\documentclass[12pt, reqno]{amsart}
\usepackage{amssymb}
\usepackage[total={6.25truein,9truein},centering]{geometry}
\usepackage[bookmarks=false]{hyperref}

\newtheorem{theorem}{Theorem}[section]
\newtheorem{lemma}[theorem]{Lemma}
\newtheorem{proposition}[theorem]{Proposition}
\newtheorem{corollary}[theorem]{Corollary}

\newcommand{\suchthat}{\;\ifnum\currentgrouptype=16 \middle\fi|\;}

\newcommand{\Z}{\mathbb{Z}}
\newcommand{\C}{\mathbb{C}}
\newcommand{\Q}{\mathbb{Q}}
\newcommand{\R}{\mathbb{R}}
\newcommand{\Gal}[1]{\operatorname{Gal}#1}

\newcommand{\N}{{\rm{N}}}

\newcommand{\F}{\mathbb{F}}

\newcommand{\beq}{\begin{equation}}
\newcommand{\eeq}{\end{equation}}

\newcommand{\ZZ}{\mathbb{Z}}

\newcommand{\FF}{\mathbb{F}}

\renewcommand\b\bullet
\renewcommand\c\times

\theoremstyle{definition}
\newtheorem{remark}[theorem]{Remark}
\newtheorem{definition}[theorem]{Definition}

\def\C{{\mathbb C}}
\def\R{{\mathbb R}}
\def\Z{{\mathbb Z}}
\def\Q{{\mathbb Q}}

\def\a{{\mathfrak a}}
\def\b{{\mathfrak b}}
\def\O_K{{\Cal{O}_{K}}}
\def\O_F{{\Cal{O}_{F}}}
\def\N_F{{\Cal{N}_{F/\Q}}}

\begin{document}

\title[Equidistribution of Gross points over rational function fields]{Equidistribution of Gross points over \\ rational function fields}

\author{Ahmad El-Guindy, Riad Masri, Matthew Papanikolas, and Guchao Zeng}

\address{Science Program, Texas A\&M University in Qatar, Doha, Qatar}
\address{Department of Mathematics, Faculty of Science, Cairo University, Giza, Egypt 12613}
\email{a.elguindy@gmail.com}

\address{Department of Mathematics, 3368 TAMU, Texas A\&M University, College Station, TX 77843-3368, U.S.A.}
\email{masri@math.tamu.edu}

\address{Department of Mathematics, 3368 TAMU, Texas A\&M University, College Station, TX 77843-3368, U.S.A.}
\email{papanikolas@tamu.edu}

\address{Science Program, Texas A\&M University in Qatar, Doha, Qatar}
\email{guchao.zeng@qatar.tamu.edu}

\date{March 27, 2020}

\thanks{This publication was made possible by the NPRP award [NPRP 9-336-1-069]
from the Qatar National Research Fund (a member of The Qatar Foundation).
The statements made herein are solely the responsibility of the authors.}

\begin{abstract}
In this paper we prove a sparse equidistribution theorem for Gross points over the
rational function field $\F_q(t)$. We apply this result to study the reduction map from CM Drinfeld modules
to supersingular Drinfeld modules.
Our proofs rely crucially on a period formula due to M.~Papikian and F.-T.\ Wei/J.~Yu,
and a Lindel\"of-type bound for central values of Rankin-Selberg $L$--functions associated
to twists of automorphic forms of Drinfeld-type by ideal class group characters.
\end{abstract}

\maketitle

\section{Introduction and statement of results}

In their ICM article \cite{MV}, Michel and Venkatesh developed a general framework concerning
sparse equidistribution problems for toric orbits of special points on Shimura varieties constructed from
quaternion algebras over totally real fields. These problems have been solved in a wide range of settings and have many striking applications.
For example, Michel~\cite{Michel} proved sparse equidistribution of Gross points corresponding to supersingular elliptic curves and used this
to study the supersingular reduction of CM elliptic curves.

In this paper we will prove a sparse equidistribution theorem for Gross points over the rational function field $k=\F_q(t)$;
see Theorem \ref{Grossed} and Corollary \ref{grossed}. Roughly speaking, given a prime $P_0$, one can construct a definite Shimura curve $X_{P_0}$
which is a disjoint union of genus zero curves. Given a Gross point $x_{0,D}$ of discriminant $D$ on $X_{P_0}$ and a subgroup
$G_D < \textrm{Pic}(\mathcal{O}_D)$ of sufficiently small index, we prove that the orbit $G_D \cdot x_{0,D}$ becomes quantitatively
equidistributed with respect to a natural probability measure on $X_{P_0}$ as $\textrm{deg}(D) \rightarrow \infty$. This result
can be viewed as part of the framework developed in \cite{MV}, in which the quaternion algebra is definite and the base field is $\mathbb{F}_q(t)$.

In analogy with \cite{Michel}, we apply Theorem \ref{Grossed}
to study the supersingular reduction of CM Drinfeld modules. Assume that $P_0$ is inert in $K=k(\sqrt{D})$. Let $\mathcal{D}(\mathcal{O}_D)$
be the set of isomorphism classes of rank 2 Drinfeld modules over $\overline{k}$ with complex multiplication by $\mathcal{O}_D$
and let $\mathcal{D}^{ss}(\overline{\F}_{P_0})$ be the set of isomorphism classes of supersingular Drinfeld modules over $\overline{\F}_{P_0}$.
If $\mathfrak{P}$ is a prime above $P_0$ in the Hilbert class field $H$ of $K$, there is a reduction map
\begin{align*}
r_{\mathfrak{P}} : \mathcal{D}(\mathcal{O}_D) \longrightarrow \mathcal{D}^{ss}(\overline{\F}_{P_0}).
\end{align*}
We will prove that the reduction map $r_{\mathfrak{P}}$ is surjective for all
$\textrm{deg}(D) \gg \textrm{deg}(P_0)$, in a quantitative sense which we make precise; see Theorem \ref{reduction}.

\subsection{Gross points over $\F_q(t)$} \label{GrossPoints} We begin by developing
some background we will need concerning Gross points over $\F_q(t)$ following
closely the discussion in Gross \cite{G}, Wei and Yu \cite[\S\S 1.1--2]{WY}.

Let $q \in \ZZ$ be a power of an odd prime, $A = \F_q[t]$ be the polynomial ring in a variable $t$, and
$k=\F_q(t)$ be the rational function field.  We let $k_{\infty}=\F_q((t^{-1}))$ be the completion of $k$ at its infinite place.  Let $P_0 \in A$ be a monic irreducible polynomial of degree $\deg(P_0) \geq 3$.
Let $D \in A$ be irreducible of odd degree, in particular ensuring that $K=k(\sqrt{D})$
is imaginary quadratic, i.e., $\infty$ does not split in $K$.  Assume further that $P_0$ is inert in $K$, and so $\chi_K(P_0)=-1$, where $\chi_K$ is the quadratic character associated to $K$.
Let $\mathcal{O}_D=A[\sqrt{D}]$ be the ring of integers,
$\mathrm{Pic}(\mathcal{O}_D)$ be the class group, and $h(D)$ be the class number of $K$, respectively.

Let $B$ be the quaternion algebra over $k$ which is ramified at $P_0$ and $\infty$, and fix a maximal $A$-order $R$ of $B$. There are finitely many equivalences classes of left ideals of $R$, and the number
$n$ of such classes is called the \textit{class number} of $R$.  Let $\{I_i\}_{i=1}^{n}$ be a set of representatives of these equivalence classes.  One associates to each $I_i$ the maximal right $A$-order of $B$ defined by
\[
R_i:=\{x \in B:~I_i x \subseteq I_i\}.
\]

Given a finite place $v$ of $k$, let $k_v$ be the completion of $k$ at $v$, let $A_v$ be the closure of $A$ in $k_v$,
let $\hat{k}$ be the finite adele ring of $k$, and let $\hat{A}=\prod_v A_v$. Also, let $\hat{R}=R \otimes_A \hat{A}$ and $\hat{B}=B \otimes_k \hat{k}$.

The left ideal classes of $R$ correspond to the orbits of $B^{\times}$ acting on the right of $\hat{R}^{\times} \backslash \hat{B}^{\times}$,
so that the class number $n$ is the number of double cosets,
\[
n=[\hat{R}^{\times} \backslash \hat{B}^{\times}/B^{\times}].
\]
The choice of representative ideals $\{I_i\}_{i=1}^{n}$ corresponds to a choice of coset representatives $\{g_i\}_{i=1}^{n}$ in
$\hat{R}^{\times} \backslash \hat{B}^{\times}$ such that
\begin{equation}\label{Bhat}
 \hat{B}^{\times} = \bigcup_{i=1}^{n}\hat{R}^{\times} g_i B^{\times}.
\end{equation}
The right order $R_i$ is given by
\[
R_i=B \cap g_i^{-1}\hat{R}g_i.
\]

An \textit{optimal embedding} $f : \mathcal{O}_D \hookrightarrow R_i$ of $\mathcal{O}_D$ into $R_i$
is a field embedding $f:K \rightarrow B$ which satisfies
\begin{equation*}
f(K) \cap g_i^{-1}\hat{R}g_i = f(\mathcal{O}_D)
\end{equation*}
in $\hat{B}$. The group $B^{\times}$ acts on the right of the set
\[
\hat{R}^{\times} \backslash \hat{B}^{\times} \times \mathrm{Hom}(K,B)
\]
by
\[
(g,f) \mapsto (gb, b^{-1}fb).
\]

There is a bijection between the set of all optimal embeddings of $\mathcal{O}_D$ into the $n$ orders $R_i$, modulo conjugation by $R_i^{\times}$,
and the classes $(g,f) \mod B^{\times}$ in the quotient space
\begin{equation}\label{quotient}
(\hat{R}^{\times} \backslash \hat{B}^{\times} \times \mathrm{Hom}(K,B))/B^{\times}
\end{equation}
satisfying
\begin{equation}\label{fid2}
f(K) \cap g^{-1}\hat{R}g = f(\mathcal{O}_D).
\end{equation}

The quotient space \eqref{quotient} can be interpreted geometrically as the set of $K$-points of a definite Shimura
curve $X_{P_0}$ defined over $k$ as follows.
There is a genus zero curve $Y$ defined over $k$ associated to $B$ whose points over $K$ are
\[
Y(K)=\{y \in B \otimes_k K:~y \neq 0, ~ \mathrm{Tr}(y)=N(y)=0\}/K^{\times}.
\]
The curve $Y$ can be described explicitly as a conic in $\mathbb{P}^2$. The group $B^{\times}$ acts on $Y$ on the right by conjugation. Moreover,
$Y(K)$ can be canonically identified with the set of embeddings $\mathrm{Hom}(K, B)$.  Define the definite Shimura curve $X_{P_0}$ by the double coset space
\[
X_{P_0}:= (\hat{R}^{\times} \backslash \hat{B}^{\times} \times Y)/B^{\times}.
\]
The decomposition \eqref{Bhat} gives a bijection to a disjoint union of genus~$0$ curves,
\[
X_{P_0} \rightarrow \bigsqcup_{i=1}^{n}X_i
\]
defined by
\[
(\hat{R}^{\times} g_i, y) \mod B^{\times} \longmapsto y \mod R_i^{\times},
\]
where $X_i:=Y/R_i^{\times}$ and $R_i^{\times} = g_i^{-1}\hat{R}^{\times}g_i^{-1} \cap B^{\times}$ is a finite group for $i=1, \ldots, n$.

\begin{definition}
A \emph{Gross point} of discriminant $D$ on $X_{P_0}(K)$
is a point $x=(g,y)$ in the image
\[
\textrm{Image}\left[\hat{R}^{\times} \backslash \hat{B}^{\times} \times Y(K) \longrightarrow X_{P_0}(K) \right]
\]
such that the embedding $f \in \mathrm{Hom}(K,B)$ corresponding to the
component $y \in Y(K)$ of $x$ satisfies \eqref{fid2}.
If the component $g \in \hat{R}^{\times} \backslash \hat{B}^{\times}/B^{\times}$ of $x$ is congruent to the double coset representative $g_i$, then $x$ lies on the component $X_i(K)$ of $X_{P_0}(K)$.
\end{definition}

Let $\mathrm{Gr}_{D,P_0}$ denote the set of Gross points of discriminant $D$. There is an action of
the group $\mathrm{Pic}(\mathcal{O}_D)$ on $\mathrm{Gr}_{D,P_0}$ given as follows.
Let $\hat{K}=K \otimes_{k} \hat{k}$ and $\hat{\mathcal{O}}_D=\mathcal{O}_D \otimes_{A} \hat{A}$. Then
$$\mathrm{Pic}(\mathcal{O}_D) \cong \hat{\mathcal{O}}_{D}^{\times} \backslash \hat{K}^{\times}/ K^{\times}.$$ Let
$x=(g,y)$ be a Gross point of discriminant $D$, and let $f \in \mathrm{Hom}(K,B)$ be the embedding corresponding
to the component $y$. This embedding induces a homomorphism $\hat{f}: \hat{K}^{\times} \rightarrow \hat{B}^{\times}$. Let
$a \in \hat{K}^{\times}$ and define
\[
x_a:=(g\hat{f}(a),y).
\]
This gives a free action of $\mathrm{Pic}(\mathcal{O}_D)$ on $\mathrm{Gr}_{D,P_0}$ which divides the set of Gross points
into two simple transitive orbits of size $h(D)$; in particular, $\#\mathrm{Gr}_{D,P_0} = 2h(D)$ (see e.g. \cite[p.~133]{G} and \cite[Lem.~1.4]{WY}).
We denote this action by $x \mapsto x^{\sigma}$ for $\sigma \in \mathrm{Pic}(\mathcal{O}_D)$.

The action of $\mathrm{Pic}(\mathcal{O}_D)$ on the set of Gross points of discriminant $D$ also translates to an action of $\mathrm{Pic}(\mathcal{O}_D)$
on the corresponding optimal embeddings. To describe this, we follow the first paragraph of \cite[p.~134]{G}.

Let $\a$ be the ideal (projective module of rank 1 in $K$) which is determined by the idele
$a \bmod {\hat{\mathcal{O}}_{D}}^{\times}$; specifically, $\a=K\cap a \hat{\mathcal{O}}_D$. Let $R_{i, \a}^{\prime}$ be the right order of the left $R_i$-module
$R_i \a$. More precisely, since $R_{i}$ is a maximal right order of $B$, we can consider the set of $n$ equivalence classes of
left $R_i$-ideals in $B$. Then, viewing $R_i \a$ as a left $R_i$-ideal, we can (as above) associate to
$R_i \a$ a maximal right order $R_{i, \a}^{\prime}$ of $B$ defined by
\begin{equation} \label{Rprime}
R_{i, \a}'=\{b \in B : R_i \a b \subseteq R_i\a\}.
\end{equation}
Since $\mathcal{O}_D$ also acts on the right of $\a$, the optimal embedding $f: \mathcal{O}_D \hookrightarrow R_i$ corresponding
to the Gross point $x$ induces an optimal embedding $f^{\prime}: \mathcal{O}_D \hookrightarrow R_{i, \a}^{\prime}$ corresponding to the Gross point $x_a$.
We denote the embedding $f^{\prime}$ by $f^{\sigma}$ and the maximal
order $R_{i, \a}^{\prime}$ by $R_{i,\sigma}^{\prime}$, where
$\sigma \in \mathrm{Pic}(\mathcal{O}_D)$ corresponds to the class of $\a$ in $\mathrm{Pic}(\mathcal{O}_D)$.

\subsection{Equidistribution of Gross points over $\F_q(t)$}

As we have seen, the Shimura curve $X_{P_0}$ is the disjoint union of $n$ genus zero curves $X_i$ defined over $k$.
Let $\mathrm{Pic}(X_{P_0})$ denote the Picard group of $X_{P_0}$. If $e_i$ denotes the class of degree $1$ in $\mathrm{Pic}(X_{P_0})$ corresponding to the component $X_i$, then we have
\[
\mathrm{Pic}(X_{P_0})=\mathbb Z e_1 \oplus \cdots \oplus \mathbb Z e_{n}.
\]
Since a Gross point $x \in \mathrm{Gr}_{D,{P_0}}$ lies on a component $X_i$, it determines
a class $e_{x}$ in $\mathrm{Pic}(X_{P_0})$ which for notational convenience we continue to denote by $x$.

For each $i$, we let $w_i := \#(R_i^{\times})/(q-1)$, and we define a probability measure on the set of divisor classes $\mathcal{S}:=\{e_i\}_{i=1}^{n}$ by
\[
\mu_{P_0}(e_i):=\frac{w_i^{-1}}{\sum_{j=1}^{n}w_j^{-1}}.
\]
We note that Gekeler~\cite[Satz (5.9)]{Gekeler83} has proved the mass formula
\begin{align}\label{mass}
  \sum_{j=1}^{n} \frac{1}{w_j} = \frac{q^{\deg(P_0)} - 1}{q^2-1},
\end{align}
which shows how the total mass grows with the degree of~$P_0$.  Furthermore, for each $i$ we have $w_i$ is either $1$ or $q+1$.

Now given a Gross point $x=x_{0,D} \in \mathrm{Gr}_{D,P_0}$ and a subgroup
$G_D < \mathrm{Pic}(\mathcal{O}_D)$, define the $G_D$-orbit
\[
G_D \cdot x_{0,D}:=\{x_{0,D}^{\sigma}:~\sigma \in G_D\} \subseteq \mathcal{S}.
\]
We want to study the distribution of the sequence of orbits $G_D \cdot x_{0,D}$ on $\mathcal{S}$ with respect to the probability measure $\mu_{P_0}$
as $\deg(D) \rightarrow \infty$. To make sense of this distribution problem, we must have
\begin{equation}\label{growth}
\# (G_D \cdot x_{0,D}) \rightarrow \infty
\end{equation}
as $\deg(D) \rightarrow \infty$. In particular, for fixed $P_0$, the size of the orbit $G_D \cdot x_{0,D}$ will eventually
exceed the number $n=n(P_0)$ of divisor classes $e_i$, which (as we have seen) equals the class number of the maximal order $R$ in $B$.
As shown in the following result,
the condition \eqref{growth} is ensured by a suitable bound on the index of $G_D$ in $\mathrm{Pic}(\mathcal{O}_D)$.

\begin{proposition}\label{indexbound}
Fix an absolute constant $0 \leq \eta < 1/2$ and let $G_D < \mathrm{Pic}(\mathcal{O}_D)$ be
a subgroup of index satisfying
\begin{equation}\label{index}
[\mathrm{Pic}(\mathcal{O}_D):G_D] \leq \lVert D \rVert^{\eta}
\end{equation}
where $\lVert D \rVert :=q^{\deg(D)}$. Then
\[
\# (G_D \cdot x_{0,D}) \gg_{\varepsilon} q^{-1} \lVert D \rVert^{(\frac{1}{2}-\eta) - \varepsilon}
\]
where the implied constant is effective. In particular, $\# (G_D \cdot x_{0,D}) \rightarrow \infty$ as $\deg(D) \rightarrow \infty$.
\end{proposition}

\begin{proof}
Let $L(\chi_K,s)$ be the $L$--function of the quadratic character $\chi_K$.
Then we have the class number formula (see e.g. \cite[\S 2.2]{CWY})
\[
L(\chi_K,1)=q\lVert D \rVert^{-1/2}h(D).
\]
By work of Weil, the Riemann hypothesis is known for the $L$--function $L(\chi_K,s)$,
which allows one to prove the following effective Siegel-type bound (see e.g. \cite[Lem.~3.3]{AT})
\[
L(\chi_K,1) \gg_{\varepsilon} \lVert D \rVert^{-\varepsilon}
\]
for any $\varepsilon > 0$. This yields the following effective lower bound for the class number,
\begin{equation}\label{siegel}
h(D) \gg_{\varepsilon} q^{-1} \lVert D \rVert^{\frac{1}{2}-\varepsilon}.
\end{equation}
In particular, from \eqref{index} and \eqref{siegel} we get
\[
\# (G_D \cdot x_{0,D}) = |G_D| = \frac{h(D)}{[\mathrm{Pic}(\mathcal{O}_D):G_D]} \gg \frac{h(D)}{\lVert D \rVert^{\eta}} \gg_{\varepsilon} q^{-1}\lVert D \rVert^{(\frac{1}{2}-\eta) - \varepsilon}.
\]
\end{proof}

We will prove that if $0 \leq \eta < 1/4$ and
\[
[\mathrm{Pic}(\mathcal{O}_D):G_D] \leq \lVert D \rVert^{\eta},
\]
then the sequence of orbits $G_D \cdot x_{0,D}$
becomes quantitatively equidistributed on $\mathcal{S}$ with respect to $\mu_{P_0}$ as $\deg(D) \rightarrow \infty$. In fact,
we will deduce our equidistribution results as a consequence of the following theorem.

\begin{theorem}\label{Grossed}
Let $q \in \Z$ be a power of an odd prime, $A=\F_q[t]$ be the polynomial ring in a variable $t$, and
$k=\F_q(t)$ be its fraction field. Let $P_0 \in A$ be a monic irreducible polynomial satisfying $\deg(P_0) \geq 3$.
Let $D \in A$ be an irreducible polynomial of odd degree such that $P_0$ is inert in the imaginary quadratic field
$K=k(\sqrt{D})$. Given a subgroup $G_D < \mathrm{Pic}(\mathcal{O}_D)$, a Gross point $x=x_{0,D} \in \mathrm{Gr}_{D,P_0}$,
and a divisor class $e_i \in \mathcal{S}$, define the counting function
\[
N_{G_D, P_{0},e_i}:=\# \{\sigma \in G_D:~ x_{0,D}^{\sigma} = e_i\}.
\]
Then for all $\varepsilon > 0$, we have
\[
\frac{N_{G_D,P_{0},e_i}}{|G_D|} = \mu_{P_0}(e_i) + O_{\varepsilon}\left([\mathrm{Pic}(\mathcal{O}_D):G_D]q^{5/4}
\lVert P_0 \rVert^{\varepsilon} \lVert D \rVert^{-1/4+\varepsilon}\right)
\]
where the implied constant in the error term is uniform in $e_i$ and effective.
\end{theorem}

The following corollary is now an immediate consequence of Theorem~\ref{Grossed}.

\begin{corollary}\label{grossed}
Let notation and assumptions be as in Theorem \ref{Grossed}.  Fix an absolute constant $0 \leq \eta < 1/4$ and let
$G_D < \mathrm{Pic}(\mathcal{O}_D)$ be a subgroup of index satisfying
\[
[\mathrm{Pic}(\mathcal{O}_D):G_D] \leq \lVert D \rVert^{\eta}.
\]
Then the sequence of orbits $G_D \cdot x_{0,D}$ becomes quantitatively equidistributed on $\mathcal{S}$
with respect to $\mu_{P_0}$ as $\deg(D) \rightarrow \infty$. In particular, we have
\begin{equation}\label{keyidentity}
\frac{N_{G_D,P_{0},e_i}}{|G_D|} = \mu_{P_0}(e_i) + O_{\varepsilon}\left(q^{5/4}\lVert P_0 \rVert^{\varepsilon} \lVert D \rVert^{-(1/4-\eta)+\varepsilon}\right).
\end{equation}
\end{corollary}

\begin{remark} In their ``Equidistribution of Subgroups Conjecture'' \cite[Conjecture 1]{MV}, Michel and Venkatesh conjecture
that for any fixed $\eta > 0$, if $x_{0,D}$ is a Heegner point of fundamental discriminant $D <0$ and $G_D < \textrm{Pic}(\mathcal{O}_D)$
is a subgroup satisfying $|G_D| \geq |D|^{\eta}$, then the sequence of orbits $G_D \cdot x_{0,D}$ becomes equidistributed with respect
to the hyperbolic measure $d\mu(z):=(3/\pi)dxdy/y^2$ on the modular curve $Y_0(1)$ as $|D| \rightarrow \infty$.
Assuming the Generalized Riemann Hypothesis, this equidistribution statement
holds for all $\eta > 1/4$. Corollary \ref{grossed} implies that the analogous equidistribution statement for Gross points
over $\F_q(t)$ holds unconditionally for all $\eta > 1/4$. This is because we are able to employ the work of Deligne and Drinfeld
on the Riemann hypothesis over function fields.
\end{remark}

\begin{remark} In our setup, we have assumed that the polynomial $D \in A$ has odd degree. This assumption is made for convenience and is not essential
in our arguments. For example, the period formula of Papikian \cite{P2} and Wei/Yu \cite{WY} which we use takes a slightly different form when
$D$ has even degree, but our proof works the same in this case. Simply put, if $D$ has even degree then the equidistribution statements in
Theorem \ref{Grossed} and Corollary \ref{grossed} hold with the same rate of convergence.
\end{remark}

\subsection{Supersingular reduction of CM Drinfeld modules} \label{ssDrinfeld}
In this section we define the reduction map from CM Drinfeld modules to supersingular Drinfeld modules, and state our result showing that this map is
surjective if $\deg(D)$ is sufficiently large compared to $\deg(P_0)$ (in a quantitative sense which we will make precise).
We first recall some definitions about Drinfeld modules (see \cite[Ch.~4]{Goss}, \cite[Ch.~3]{Thakur} for more details).

Given a field $L/\FF_q$, the $q$-th power Frobenius endomorphism $\tau : L \to L$ generates an $\FF_q$-subalgebra of the endomorphism algebra of the additive group of $L$, which we denote by $L\{\tau\}$.  The ring $L\{\tau\}$ is the ring of twisted polynomials in $\tau$ with coefficients in $L$, and it is subject to the relation $\tau c = c^q \tau$ for all $c \in L$.  Fixing an $\FF_q$-algebra homomorphism $\iota : A \to L$, a \emph{Drinfeld module of rank~$r$ over~$L$} is an $\FF_q$-algebra homomorphism
\[
\phi: A \to L\{\tau\},
\]
so that for all $a \in A$,
\[
  \phi_a = \iota(a) + b_1 \tau + \cdots + b_{r \deg(a)} \tau^{r \deg(a)}, \quad b_{r \deg(a)} \neq 0.
\]
For two Drinfeld modules $\phi$, $\psi$ over $L$, a \emph{morphism} $u : \phi \to \psi$ over $L$ is a twisted polynomial $u \in L\{\tau\}$, so that
\[
 u \phi_a = \psi_a u, \quad \forall\,a \in A.
\]
If $\phi$ and $\psi$ have different ranks, then the only possible morphism between them is the zero morphism.

Henceforth, we will primarily focus on Drinfeld modules of rank~$2$.  Since a Drinfeld module $\phi$ is completely determined by its value $\phi_{t}$, we can fix a rank $2$ Drinfeld module over $L$ by setting,
\[
  \phi_t := \iota(t) + g \tau + \Delta \tau^2, \quad g, \Delta \in L,
\]
in which case the $j$-invariant $j(\phi) := g^{q+1}/\Delta$ is an isomorphism invariant of $\phi$ over an algebraic closure $\overline{L}$.

Let $\mathcal{D}(\mathcal{O}_D)$ be the set of $\overline{k}$-isomorphism classes of Drinfeld modules of rank~$2$ over $\overline{k}$ with complex multiplication (CM) by $\mathcal{O}_D$.  Here we take $\iota : A \hookrightarrow \overline{k}$ to be the inclusion map, and a Drinfeld module $\phi$ of rank~$2$ has complex multiplication by $\mathcal{O}_D$ if there is an extension of $\phi : \mathcal{O}_D \to L\{\tau\}$, coinciding with $\phi$ on $A$.  By the theory of complex multiplication \cite[Ch.~7]{Goss}, each isomorphism class in $\mathcal{D}(\mathcal{O}_D)$ is represented by a sign-normalized Drinfeld-Hayes module (of rank~$1$ as a Drinfeld $\mathcal{O}_D$-module, see~\cite[Ch.~7]{Goss}, \cite{Hayes79}).  Such a Drinfeld-Hayes module is defined over the Hilbert class field $H$ of $K$, i.e.\ the maximal abelian extension of $K$ that is everywhere unramified and totally split at the unique infinite place of $K$.  The group $\mathrm{Gal}(H/K) \cong \mathrm{Pic}(\mathcal{O}_D)$ acts simply transitively
on $\mathcal{D}(\mathcal{O}_D)$.  In particular, given a Drinfeld-Hayes module $\phi \in \mathcal{D}(\mathcal{O}_D)$,
we have
\[
\mathcal{D}(\mathcal{O}_D) = \{[\phi^{\sigma}] : \sigma \in \mathrm{Pic}(\mathcal{O}_D)\},
\]
where $\phi^{\sigma}$ denotes the action of $\mathrm{Pic}(\mathcal{O}_D)$ on $\phi$. There are $h(D)$ such isomorphism classes~\cite[Cor.~5.13]{Hayes79}.

Moreover, $\phi$ is defined over the ring of integers $\mathcal{O}_H$ of $H$, so that
\[
\phi_t = \iota(t) + b_1 \tau + \tau^{2}, \quad b_1 \in \mathcal{O}_H.
\]
Thus if $\mathfrak{P}$ is a prime ideal of $\mathcal{O}_H$ and $\bar{\iota} : A \to \mathcal{O}_H/\mathfrak{P}$ is the map induced by the inclusion $A \hookrightarrow \mathcal{O}_H$, we can define the reduction $\overline{\phi}$ of $\phi$,
\[
\overline{\phi}_t = \overline{\iota}(t) + \overline{b}_1 \tau +  \tau^{2}, \quad \overline{b}_1 \in \mathcal{O}_H/\mathfrak{P},\ \overline{b}_1 \equiv b_1 \pmod{\mathfrak{P}}.
\]
which is a Drinfeld module of rank~$2$ over $\mathcal{O}_H/\mathfrak{P}$.  Letting $P_0 \in A$ be the unique monic generator of $\mathfrak{P} \cap A$, we say that the reduction $\overline{\phi}$ is \emph{supersingular} if
\[
\overline{\phi}_{P_0} = \tau^{2 \deg(P_0)} \in (\mathcal{O}_H/\mathfrak{P} )\{\tau\},
\]
which is equivalent to $\overline{\phi}_{P_0}$ being purely inseparable as a map on $\mathbb{G}_a$ (see~\cite[\S 4]{Gekeler91}).  An equivalent condition is simply to check that the coefficient of $\tau^{\deg(P_0)}$ in $\phi_{P_0}$ is divisible by $P_0$, and this can be determined effectively (e.g., see~\cite[Cor~8.2]{EP13}).  Moreover, one can determine supersingular $j$-invariants via recursive identities on Drinfeld modular forms, using a construction of Cornelissen~\cite{Cornelissen99}.

If $\phi$ has supersingular reduction at $P_0 \nmid D$, then the endomorphism algebra of $\overline{\phi}$ is a maximal order in the quaternion algebra $B$~\cite[Thm.~2.9, Thm.~4.3]{Gekeler91}.  Moreover, as this induces an embedding $\mathcal{O}_D \hookrightarrow B$, it follows that $P_0$ must be inert in $\mathcal{O}_D$~\cite[Lem.~2.2]{P2}.

Fix now $P_0 \in A$ monic, irreducible, and inert in $K$.  We set $\F_{P_0}:=A/(P_0)$ and let $\mathcal{D}^{ss}(\overline{\mathbb{F}}_{P_0})$ be the
set of isomorphism classes of supersingular Drinfeld modules over $\overline{\mathbb{F}}_{P_0}$.  There are $n$ such isomorphism classes
\[
\mathcal{D}^{ss}(\overline{\mathbb{F}}_{P_0}) = \{e_1, \ldots, e_n \},
\]
and we further have bijective correspondences
\[
 \{X_1, \ldots, X_n\} \longleftrightarrow \{e_1, \ldots, e_n \} \longleftrightarrow \{I_1, \ldots, I_n\} \longleftrightarrow \{R_1, \ldots, R_n\}
\]
such that $\mathrm{End}(e_i) \cong R_i$ (see~\cite[\S\S 3--4]{Gekeler91}, \cite[Thm.~2.6]{P2}).  Letting $\mathfrak{P}$ be a prime above $P_0$ in $H$, we obtain a reduction map
\[
r_{\mathfrak{P}}:  \mathcal{D}(\mathcal{O}_D) \longrightarrow \mathcal{D}^{ss}(\overline{\mathbb{F}}_{P_0}).
\]

\begin{theorem}\label{reduction}
The reduction map $r_{\mathfrak{P}}$ is surjective if $\lVert D \rVert \gg_{\varepsilon, q} \lVert P_0 \rVert^{4 + \varepsilon}$, where
the implied constant in $\gg_{\varepsilon,q}$ is effective.
\end{theorem}

\begin{remark}
Other aspects of CM-liftings of supersingular Drinfeld modules have been studied previously.  Schweizer~\cite[Prop.~13]{S} investigated how Drinfeld modules with CM by the order of conductor $P_0$ in $\FF_{q^2}[t]$ are in bijection with supersingular Drinfeld modules modulo $P_0$.  Cojocaru and Papikian~\cite[\S 5.2]{CP} have shown how to perform CM-liftings for supersingular Drinfeld modules of arbitrary rank.  However, in these cases one varies by orders in a fixed imaginary extension of $\mathbb{F}_q(t)$ rather than by maximal orders in varying imaginary quadratic extensions of $\mathbb{F}_q(t)$ as in Theorem~\ref{reduction}.
\end{remark}

\begin{remark} Liu, Young, and the second  author \cite[Cor. 1.3]{LMY} proved a result similar to Theorem \ref{reduction} for the reduction map from CM elliptic curves to
supersingular elliptic curves.
\end{remark}

\section{Deducing Theorem~\ref{reduction} from Theorem \ref{Grossed}}

In this section we show how Theorem~\ref{reduction} follows from Theorem~\ref{Grossed} and Corollary~\ref{grossed}.
To do this we show that the image $r_{\mathfrak{P}}(\mathcal{D}(\mathcal{O}_D))$ can be identified with the set of Gross points of discriminant $D$.
Then since Corollary~\ref{grossed} implies that this image becomes equidistributed among the classes $\mathcal{D}^{ss}(\overline{\mathbb{F}}_{P_0})$
as $\deg(D) \rightarrow \infty$, every class $e \in \mathcal{D}^{ss}(\overline{\mathbb{F}}_{P_0})$ is hit at least once
if $\deg(D)$ is sufficiently large, in a quantitative sense which can be made precise using \eqref{keyidentity}.

Since each class in $\mathcal{D}(\mathcal{O}_D)$ is represented by a sign-normalized Drinfeld-Hayes module~$\phi$, if
$r_{\mathfrak{P}}(\phi)=e$, then one obtains an embedding of the corresponding endomorphism rings,
\begin{equation} \label{fPphidef}
f_{\mathfrak{P},\phi}: \mathrm{End}_H(\phi) \cong \mathcal{O}_D \hookrightarrow \mathrm{End}_{\overline{\mathbb{F}}_{P_0}}(e) \cong R_{e}.
\end{equation}
The induced embedding $f_{\mathfrak{P},\phi} : K \hookrightarrow B$ is necessarily optimal.  Indeed, this amounts to having the equality of subalgebras of~$B$,
 \[
f_{\mathfrak{P},\phi}(K) \cap g_e^{-1}\hat{R}g_e = f_{\mathfrak{P},\phi}(\mathcal{O}_D),
\]
where $g_e$ is the coset representative in $\hat{R}^{\times}\backslash \hat{B}^{\times}$ corresponding to the ideal $I_e$; the left-hand side is always contained in the right, and by construction via sign-normalized modules, the right-hand side is contained in the left.  We note that the optimality of embeddings in the context of CM elliptic curves was proved by Bertolini and Darmon \cite[Prop.~4.1]{BD}.

Now recall that the $R_e^{\times}$-conjugacy class of the optimal embedding
$f_{\mathfrak{P},\phi}$ corresponds to a Gross point $x_{\mathfrak{P}, \phi}=(g_e, f_{\mathfrak{P},\phi})$ of discriminant $D$, and using the $\mathrm{Pic}(\mathcal{O}_D)$-action on Gross points described \S\ref{GrossPoints},
the $(R_{e,\sigma})^{\times}$-conjugacy class of the optimal embedding $f_{\mathfrak{P},\phi}^{\sigma}$ for $\sigma \in \mathrm{Pic}(\mathcal{O}_D)$ corresponds to the Gross point
$x_{\mathfrak{P}, \phi}^{\sigma}=(g_{e,\sigma}, f_{\mathfrak{P},\phi}^{\sigma})$ of discriminant $D$. We are now led to the following crucial equivariance result.

\begin{proposition}
Given a sign-normalized Drinfeld-Hayes module $\phi \in \mathcal{D}(\mathcal{O}_D)$, we have $f_{\mathfrak{P},\phi}^{\sigma}=f_{\mathfrak{P},\phi^{\sigma}}$ for $\sigma \in \mathrm{Pic}(\mathcal{O}_D)$.
\end{proposition}

\begin{proof}
Bertolini and Darmon \cite[Lem.~4.2]{BD} proved this result in the context of elliptic curves.  We first recall the definition of $\phi^{\sigma}$: by identifying $\sigma$ with an element of $\mathrm{Gal}(H/K)$, if $\phi_a = a + b_1 \tau + \cdots + b_{\ell} \tau^\ell$, then
\[
  \phi^{\sigma}_a := a + b_1^{\sigma} \tau + \cdots + b_{\ell}^{\sigma} \tau^{\ell}.
\]
It is clear that $\phi^{\sigma}$ is also a sign-normalized Drinfeld-Hayes module defined over $\mathcal{O}_H$.

The definition of $\phi^{\sigma}$ is compatible with an action of $\mathrm{Pic}(\mathcal{O}_D)$ in the following way. As originally defined by Hayes~\cite[\S 4.9]{Goss}, \cite[\S\S 3--5]{Hayes79}, for a fixed integral ideal $\mathfrak{a} \subseteq \mathcal{O}_D$ we let
\[
  \mu_{\mathfrak{a}} := \textup{unique monic generator of the left ideal $H\{\tau\} \cdot \phi(\mathfrak{a})$ of $H\{\tau\}$,}
\]
which is well-defined since $H\{\tau\}$ is a left principal ideal domain.  Then there is a unique sign-normalized Drinfeld module $\phi^{\mathfrak{a}}$ such that $\mu_{\mathfrak{a}} : \phi \to \phi^{\mathfrak{a}}$ is a morphism, and furthermore $\mu_{\mathfrak{a}} \in \mathcal{O}_H\{\tau\}$ by~\cite[Prop.~7.5]{Hayes79}.  If the class of $\mathfrak{a}$ in $\mathrm{Pic}(\mathcal{O}_D)$ corresponds to the Galois element $\sigma \in \Gal(H/K)$ via class field theory, then by~\cite[Prop.~8.1]{Hayes79},
\[
  \phi^{\mathfrak{a}}_a = \phi^{\sigma}_a, \quad \forall\, a \in A.
\]

Now suppose that $r_{\mathfrak{P}}(\phi) = e$.  For simplicity, we will write $\mathbb{F} = \mathbb{F}_{P_0}$, and we will assume that $R_e = \mathrm{End}_{\overline{\F}}(e) \subseteq \F\{ \tau\}$.  Thus the embedding $f_{\mathfrak{P},\phi} : \mathcal{O}_D \hookrightarrow R_e$ defined in~\eqref{fPphidef} takes values in $\F\{ \tau\}$. If we fix again an integral ideal $\mathfrak{a} \subseteq \mathcal{O}_D$, we take
\[
  I = R_e \cdot f_{\mathfrak{P},\phi}(\mathfrak{a}) = R_e \cdot e(\mathfrak{a}),
\]
which is a left ideal of $R_e$.  We let
\[
  \nu_{\mathfrak{a}} := \textup{unique monic generator of the left ideal $\overline{\F} \{\tau\} \cdot I$ of $\overline{\F} \{\tau\}$,}
\]
and by J.-K.\ Yu~\cite[\S 2]{YuJK}, there is a unique Drinfeld module $e^{I}$ over $\overline{\F}$ such that $\nu_{\mathfrak{a}} : e \to e^{I}$ is a morphism.  Moreover, by \cite[Prop.~2]{YuJK} we have
\[
  \mathrm{End}_{\overline{\F}} (e^{I}) = \{ b \in B : I b \subseteq I \} = R_{e,\mathfrak{a}}' = R_{e,\sigma}',
\]
the right order of $I$ in $B$, where $R_{e,\mathfrak{a}}'$ is defined in~\eqref{Rprime}.

We claim that for our fixed ideal $\mathfrak{a} \subseteq \mathcal{O}_D$, we have
\[
  \nu_{\mathfrak{a}} \equiv \mu_{\mathfrak{a}} \pmod{\mathfrak{P}},
\]
i.e., $\nu_{\mathfrak{a}}$ is obtained from $\mu_{\mathfrak{a}}$ by reducing each of its coefficients modulo $\mathfrak{P}$.  Letting $\overline{\mu}_{\mathfrak{a}}$ denote reduction of $\mu_{\mathfrak{a}}$ modulo $\mathfrak{P}$, we observe that $\overline{\mu}_{\mathfrak{a}} \in e(\mathfrak{a}) \subseteq I$, since $r_{\mathfrak{P}}(\phi) = e$.  Now the zeros of $\mu_{\mathfrak{a}}$, as a function on $\overline{H}$, comprise the $\mathfrak{a}$-torsion of $\phi$ (see~\cite[\S 4.9]{Goss}), and as such
\[
  \deg_{\tau} (\mu_{\mathfrak{a}}) = \dim_{\F_q}(\mathcal{O}_D/\mathfrak{a}).
\]
By the same token, every element of $I$ must vanish at the $\mathfrak{a}$-torsion of $e$, and so
\[
  \deg_{\tau} (\nu_{\mathfrak{a}}) \geq \dim_{\F_q}(\mathcal{O}_D/\mathfrak{a}).
\]
Since $\overline{\mu}_{\mathfrak{a}} \neq 0$, it follows that $\overline{\mu}_{\mathfrak{a}} = \nu_{\mathfrak{a}}$.  As a consequence, we see that
\[
r_{\mathfrak{P}}(\phi^{\mathfrak{a}}) = r_{\mathfrak{P}}(\phi^{\sigma}) = e^{I}.
\]
Therefore,
\[
  f_{\mathfrak{P},\phi^{\sigma}} : \mathcal{O}_D \hookrightarrow \mathrm{End}_{\overline{\F}}(e^I) = R_{e,\sigma}',
\]
but this coincides with $f_{\mathfrak{P},\phi}^{\sigma}$ defined in \S\ref{GrossPoints}.
\end{proof}

\begin{proof}[Proof of Theorem \ref{reduction}]
By the preceding discussions, we have bijective correspondences
\begin{align*}
r_{\mathfrak{P}}(\mathcal{D}(\mathcal{O}_D)) &\longleftrightarrow \{r_{\mathfrak{P}}(\phi^{\sigma}):~\sigma \in \mathrm{Pic}(\mathcal{O}_D)\}\\
& \longleftrightarrow \{f_{\mathfrak{P},\phi^{\sigma}}:~\sigma \in \mathrm{Pic}(\mathcal{O}_D)\} \\
& \longleftrightarrow  \{f_{\mathfrak{P},\phi}^{\sigma}:~\sigma \in \mathrm{Pic}(\mathcal{O}_D)\} \\
& \longleftrightarrow  \{x_{\mathfrak{P},\phi}^{\sigma}:~\sigma \in \mathrm{Pic}(\mathcal{O}_D)\}.
\end{align*}
Since the components $\{X_1, \ldots, X_n\}$ of the definite Shimura curve $X_{P_0}$, and hence the divisor classes $\mathcal{S}$, are identified with the isomorphism classes of supersingular Drinfeld modules
$\{e_1, \ldots, e_n\}$, and by Corollary \ref{grossed} the $\mathrm{Pic}(\mathcal{O}_D)$-orbit $\{x_{\mathfrak{P},\phi}^{\sigma}: \sigma \in \mathrm{Pic}(\mathcal{O}_D)\}$
of the Gross point $x_{\mathfrak{P},\phi}$ becomes equidistributed with respect to the probability measure $\mu_{P_0}$ on $\mathcal{S}$
as $\deg(D)\rightarrow \infty$,
it follows from the above correspondences that the reduction map $r_{\mathfrak{P}}$ is surjective if $\deg(D)$ is sufficiently large. More precisely,
since
\[
\lVert D \rVert^{1/2 -\varepsilon} \ll_{q, \varepsilon} h(D) \ll_{q, \varepsilon} \lVert D \rVert^{1/2 + \varepsilon}
\]
and $\mu_{P_0}(e_i) \asymp \lVert P_0 \rVert^{-1}$ (see \eqref{mass}), then by \eqref{keyidentity} with $\eta=0$
we have $N_{\mathrm{Pic}(\mathcal{O}_D),P_0, e_i} > 0$ (and hence that $r_{\mathfrak{P}}$ is surjective) if $\lVert D\rVert \gg_{\varepsilon, q} \lVert P_0 \rVert^{4 + \varepsilon}$, and we are done.
\end{proof}

\section{Rankin-Selberg $L$--functions}

In this section we deduce some facts we will need on automorphic forms of Drinfeld-type and Rankin-Selberg $L$--functions from
 \cite[\S\S 3--4]{CWY}.

Let $M(P_0)$ (resp.~$S(P_0)$) be the space of automorphic forms (resp.~cusp forms) of Drinfeld type for $\Gamma_0(P_0)$ with
Petersson inner product $\langle f, f\rangle_{P_0}$. Our assumption that $\textrm{deg}(P_0) \geq 3$ ensures that the dimension
of the space $S(P_0)$ is positive (see \eqref{dimensionformula}). For
each place $v$ of $k$, let $T_v$ be the Hecke operator on $M(P_0)$ corresponding to $v$.
Let $f \in M(P_0)$ be a normalized Hecke eigenform, and let $\lambda_v(f) \in \R$
be the Hecke eigenvalue corresponding to $T_v$. Then, we have $\lambda_\infty(f)=1$, and for each finite place $v$ of
$k$, we have
\begin{itemize}
\item $|\lambda_v(f)| \leq 2 q_{v}^{1/2}$ if $v \neq P_0$,
\item $\lambda_{P_0}(f)^2=1$,
\end{itemize}
where $q_v$ is the cardinality of the residue field $\mathbb{F}_v$ of $k_v$.
The first inequality is the Ramanujan bound, and the second inequality holds since
$f$ has trivial central character (under our assumptions, any form $f \in M(P_0)$ has trivial central character $w_f$, since
$w_f$ can be identified with a character on the ideal class group $\mathrm{Pic}(A)$, and the latter group is trivial; see the last
paragraph of \cite[\S 3.1]{CWY}).

Let $f \in S(P_0)$ be a normalized newform and $\chi$ be a character of
$\mathrm{Pic}(\mathcal{O}_D)$. Define the Rankin-Selberg $L$--function
\[
L(f \times \chi, s):=\prod_{v}L_v(f \times \chi,s),
\]
where the product is over all places $v$ of $k$, and the local factors are
given by
\[
L_v(f \times \chi,s):=\prod_{1 \leq j, k \leq 2}(1-\alpha_v^{(j)}(f)c_v^{(k)}(\chi)q_v^{-(s+\frac{1}{2})})^{-1},
\]
where the numbers
$\alpha_v^{(j)}(f), c_v^{(k)}(\chi)$ are determined as follows:
\begin{itemize}
\item when $v \nmid P_0 \infty$, $\alpha_v^{(1)}(f)$ and $\alpha_v^{(2)}(f)$ are the two complex conjugate roots
of the quadratic polynomial
\[
X^2-\lambda_v(f)X+q_v;
\]
\item when $v \mid P_0 \infty$, we have $\alpha_v^{(1)}(f)=\lambda_v(f)$ and $\alpha_v^{(2)}(f)=0$;
\item when $v=w_1w_2$ splits in $K/k$, $c_v^{(k)}(\chi)=\chi(w_k)$ for $k=1,2$;
\item when $v$ is inert in $K/k$, $c_v^{(k)}(\chi)=(-1)^{k}\sqrt{\chi(w)}$ where $w\mid v$, for $k=1,2$;
\item when $v$ is ramified in $K/k$, $c_v^{(1)}(\chi)=\chi(w)$ where $w\mid v$, and $c_v^{(2)}(\chi)=0$.
\end{itemize}
Note that we have normalized the $L$--function $L(f \times \chi,s)$ so that the central value occurs at $s=1/2$.

Under our assumptions, the place $v=\infty$ is ramified in $K/k$, hence using that $\lambda_{\infty}(f)=1$, $\chi(K_{\infty}^{\times})=1$,
and $q_{\infty}=q$, we get
\[
L_{\infty}(f \times \chi,s)=(1-q^{-(s+\frac{1}{2})})^{-1}.
\]
Similarly, under our assumptions, the place $v=P_0$ is inert in $K/k$, hence using that $\lambda_{P_0}(f)^2=1$, we get
\[
L_{P_0}(f \times \chi,s)=(1-\chi(w_{P_0})q_{P_0}^{-2(s+\frac{1}{2})})^{-1}
\]
where $w_{P_0} \mid P_0$. Therefore, we have the Euler product
\begin{multline}\label{euler}
L(f \times \chi, s)
=(1-q^{-(s+\frac{1}{2})})^{-1}(1-\chi(w_{P_0})q_{P_0}^{-2(s+\frac{1}{2})})^{-1} \\
{} \times \prod_{v \nmid P_{0}\infty}\prod_{1 \leq j, k \leq 2}(1-\alpha_v^{(j)}(f)c_v^{(k)}(\chi)q_v^{-(s+\frac{1}{2})})^{-1}.
\end{multline}

\section{Bounding the central value $L(f \times \chi, 1/2)$}

In this section we will prove the following Lindel\"of-type bound for the
central value $L(f \times \chi, 1/2)$.

\begin{theorem}\label{lindelof}
We have
\[
|L(f \times \chi, 1/2)|  \leq \exp\left( \frac{3m}{2\log_q(m/2)} + O(\log(\log(\log_q(m/2) + 1))q^{1/2}m^{1/2}) \right)
\]
where $m \in \Z^{+}$ satisfies $m=O(\deg(P_0) + \deg(D))$.
\end{theorem}

\begin{proof} We adapt the argument in \cite[Thm.~3.3]{AT} used to bound the degree one $L$--function $L(\chi_K,s)$ at $s=1/2$.

By work of Deligne \cite{De} and Drinfeld \cite{Dr}, the $L$--function $L(f \times \chi,s)$ is a polynomial of degree $m=O(\deg(P_0) + \deg(D))$
in $q^{-s}$ and has zeros only on the critical line $\mathrm{Re}(s)=1/2$.
Therefore, we can write
\[
L(f \times \chi, s) = a \prod_{k=1}^{m}(1-\alpha_k q^{\frac{1}{2}-s}) = a q^{-ms}\prod_{k=1}^{m}(q^s - \alpha_k q^{1/2})
\]
for complex numbers $a \neq 0$ and $\alpha_k$ with $|\alpha_k|=1$. Taking logarithmic derivatives yields
\[
\frac{L^{\prime}}{L}(f \times \chi,s) = \log(q)\left( -m + \sum_{k=1}^{m}\frac{1}{1-\alpha_kq^{\frac{1}{2}-s}}\right).
\]

Define
\[
F(s):=\sum_{k=1}^{m}\mathrm{Re}\left(\frac{1}{1-\alpha_kq^{\frac{1}{2}-s}} - \frac{1}{2}\right).
\]
Then for $s \in \R$, we have
\begin{equation}\label{logder}
\frac{L^{\prime}}{L}(f \times \chi,s) = \log(q) \left(-\frac{m}{2} + F(s) \right).
\end{equation}
Let $s_0 \in \R$ be such that $1/2 < s_0 < 1/2 + 1/\log(q)$. Then integrating from $1/2$ to $s_0$ gives

\begin{multline}\label{logdifference}
 \log|L(f \times \chi, 1/2)| - \log|L(f \times \chi, s_0)| \\
 =\frac{m}{2}\log(q)(s_0 -\frac{1}{2})
 - \log(q)\sum_{k=1}^{m} \int_{1/2}^{s_0}\mathrm{Re}\left(\frac{1}{1-\alpha_kq^{\frac{1}{2}-s}} - \frac{1}{2}\right)ds.
\end{multline}

To estimate the second term on the RHS of \eqref{logdifference},
we use the following lemma (see \cite[Lem.~3.1]{AT}).

\begin{lemma}\label{ASLemma1}
Let $\theta$ and $0 < t < 1$ be real numbers. Then
\[
\int_{0}^{t}\mathrm{Re}\left(\frac{1}{1-e^{-x-i\theta}}-\frac{1}{2}\right)dx
\geq 2 \cdot \frac{1+e^{-t}}{1-e^{-t}}\cdot \mathrm{Re}\left(\frac{1}{1-e^{-t-i\theta}} -\frac{1}{2}\right).
\]
\end{lemma}

Since $|\alpha_k|=1$, we can write
\[
\alpha_k q^{\frac{1}{2}-s} = e^{-\log(q)(s - \frac{1}{2})-i\theta_k}
\]
for some $\theta_k \in \R$. Make the change of variables $x=\log(q)(s-1/2)$ to get
\[
\int_{1/2}^{s_0}\mathrm{Re}\left(\frac{1}{1-\alpha_kq^{\frac{1}{2}-s}} - \frac{1}{2}\right)ds
=\frac{1}{\log(q)}\int_{0}^{\log(q)(s_0 -\frac{1}{2})}\mathrm{Re}\left(\frac{1}{1-e^{-x -i\theta_k}} - \frac{1}{2}\right)dx.
\]
By Lemma \ref{ASLemma1} we have
\[
\int_{0}^{\log(q)(s_0 -\frac{1}{2})}\mathrm{Re}\left(\frac{1}{1-e^{-x -i\theta_k}} - \frac{1}{2}\right)dx
\geq 2\cdot \frac{1+q^{\frac{1}{2}-s_0}}{1-q^{\frac{1}{2}-s_0}}\cdot \mathrm{Re}\left(\frac{1}{1-\alpha_kq^{\frac{1}{2}-s_0}} - \frac{1}{2}\right).
\]
Then applying this bound in \eqref{logdifference} gives
\begin{equation}\label{logdifference2}
 \log|L(f \times \chi, 1/2)| - \log|L(f \times \chi, s_0)| \leq \frac{m}{2}\log(q)(s_0 - \frac{1}{2})
- 2\cdot  \frac{1+q^{\frac{1}{2}-s_0}}{1-q^{\frac{1}{2}-s_0}} \cdot F(s_0).
\end{equation}

Define $h:=\lceil \log_q(m/2)\rceil$. For $\mathrm{Re}(s) > 0$, the integral
\[
\frac{1}{2\pi i}\int_{2}^{2 + \frac{2\pi i}{\log(q)}}-\frac{L^{\prime}}{L}(f \times \chi, s+ w)\frac{q^{hw}q^{-w}}{(1-q^{-w})^2}dw
\]
can be computed in two different ways, first by expanding
\[
\frac{L^{\prime}}{L}(f \times \chi, s) = \sum_{n=1}^{\infty} \frac{c_{f,K}(n)}{q^{ns}}
\]
and integrating term by term, and second by analytically continuing to the left and picking up the residues
at the poles. There is a double pole at $w=0$, and simple poles at the values of $w$ for which $q^{s+w}=\alpha_k q^{1/2}$.
This yields the identity
\begin{align*}
-\log(q)^{-2}\sum_{n=1}^{h}\frac{c_{f,K}(n)\log(q^{h-n})}{q^{ns}} & = h \log(q)^{-1}\frac{L^{\prime}}{L}(f \times \chi, s)
+ \log(q)^{-2}\left(\frac{L^{\prime}}{L}(f \times \chi, s) \right)^{\prime}\\
 & \quad + \sum_{k=1}^{m}\frac{(\alpha_k q^{\frac{1}{2}-s})^{h} \alpha_k^{-1} q^{s-\frac{1}{2}}}{(1-\alpha_k^{-1}q^{s-\frac{1}{2}})^2}.
\end{align*}
Integrate from $s_0$ to $\infty$, take real parts, and multiply by $\log(q)^2$ to obtain the identity
\begin{align}\label{logdifference3}
h\log(q)\log|L(f \times \chi, s_0)| & = -\frac{L^{\prime}}{L}(f \times \chi, s_0) + \log(q)^{-1}\sum_{n=1}^{h}\frac{c_{f,K}(n)\log(q^{h-n})}{nq^{ns_0}}\\
& \quad + \log(q)^{2} \sum_{k=1}^{m}
\int_{s_0}^{\infty}\mathrm{Re}\left(\frac{(\alpha_k q^{\frac{1}{2}-s})^{h} \alpha_k^{-1} q^{s-\frac{1}{2}}}{(1-\alpha_k^{-1}q^{s-\frac{1}{2}})^2}\right)ds.\notag
\end{align}

To estimate the third term on the RHS of \eqref{logdifference3}, we use the following lemma (see \cite[Lem.~3.2]{AT}).

\begin{lemma}\label{ASLemma2} For $s > s_0 > 1/2$ we have
\[
\left| \frac{\alpha_k^{-1} q^{s-\frac{1}{2}}}{(1-\alpha_k^{-1}q^{s-\frac{1}{2}})^2}\right| \leq \log(q)^{-1} (s_0 - \frac{1}{2})^{-1}
\mathrm{Re}\left(\frac{1}{1-\alpha_kq^{\frac{1}{2}-s_0}}\right).
\]
\end{lemma}

By Lemma \ref{ASLemma2} we have
\[
\int_{s_0}^{\infty}\mathrm{Re}\left(\frac{(\alpha_k q^{\frac{1}{2}-s})^{h} \alpha_k^{-1} q^{s-\frac{1}{2}}}{(1-\alpha_k^{-1}q^{s-\frac{1}{2}})^2}\right)ds
\leq h^{-1} \log(q)^{-2} (s_0 - \frac{1}{2})^{-1} \mathrm{Re}\left(\frac{1}{1-\alpha_kq^{\frac{1}{2}-s_0}}\right)q^{h(\frac{1}{2}-s_0)}.
\]
Then applying this bound in \eqref{logdifference3} gives
\begin{multline}\label{s_0Bound}
\log|L(f \times \chi, s_0)|  \leq -h^{-1}\log(q)^{-1}\frac{L^{\prime}}{L}(f \times \chi, s_0)
 + h^{-1} \log(q)^{-2}\sum_{n=1}^{h}\frac{c_{f,K}(n)\log(q^{h-n})}{nq^{ns_0}} \\
 {}+ h^{-2}\log(q)^{-1}(s_0 - \frac{1}{2})^{-1} \left(F(s_0) + \frac{m}{2} \right)q^{h(\frac{1}{2}-s_0)}.
\end{multline}

We now apply \eqref{s_0Bound} in \eqref{logdifference2}, then use \eqref{logder} to get
\begin{align*}
\log|L(f \times \chi, 1/2)| & \leq \frac{m}{2}\left(\log(q)(s_0 - \frac{1}{2}) + h^{-1} + h^{-2}\log(q)^{-1}(s_0 - \frac{1}{2})^{-1}q^{h(\frac{1}{2}-s_0)}\right)\\
& \quad + F(s_0)\left((s_0 - \frac{1}{2})^{-1}\frac{q^{h(\frac{1}{2}-s_0)}}{h^2 \log(q)} - 2 \cdot\frac{1+q^{\frac{1}{2}-s_0}}{1-q^{\frac{1}{2}-s_0}} - h^{-1}\right)\\
& \quad +  h^{-1} \log(q)^{-2}\sum_{n=1}^{h}\frac{c_{f,K}(n)\log(q^{h-n})}{nq^{ns_0}}.
\end{align*}
Choose $s_0=1/2 + 1/h\log(q)$. Then
\[
 h^{-2}\log(q)^{-1}(s_0 - \frac{1}{2})^{-1}q^{h(\frac{1}{2}-s_0)} < h^{-1}
\]
and
\[
(s_0 - \frac{1}{2})^{-1}\frac{q^{h(\frac{1}{2}-s_0)}}{h^2 \log(q)} - 2\cdot \frac{1+q^{\frac{1}{2}-s_0}}{1-q^{\frac{1}{2}-s_0}} - h^{-1} < 0.
\]
Since $F(s_0) > 0$, it follows that
\begin{equation}\label{centralbound}
\log|L(f \times \chi, 1/2)| \leq \frac{3m}{2h} +  h^{-1} \log(q)^{-2}\sum_{n=1}^{h}\frac{c_{f,K}(n)\log(q^{h-n})}{nq^{ns_0}}.
\end{equation}

We next bound the coefficients $c_{f,K}(n)$. Taking the logarithmic derivative of the Euler product \eqref{euler} yields
\begin{align*}
\frac{L^{\prime}}{L}(f \times \chi, s) & = -\log(q)\frac{q^{-(s+\frac{1}{2})}}{1-q^{-(s+\frac{1}{2})}} \\
& \quad
- 2\log(q_{P_0})\frac{\chi(w_{P_0})q_{P_0}^{-2(s+\frac{1}{2})}}{1-\chi(w_{P_0})q_{P_0}^{-2(s+\frac{1}{2})}}\\
& \quad - \sum_{v \nmid P_0 \infty}\log(q_v)\sum_{1 \leq j, k \leq 2}
\frac{\alpha_v^{(j)}(f)c_v^{(k)}(\chi)q_v^{-(s+\frac{1}{2})}}{1-\alpha_v^{(j)}(f)c_v^{(k)}(\chi)q_v^{-(s+\frac{1}{2})}}.
\end{align*}
Since the eigenvalues $\lambda_v(f)$ are real, the two complex conjugate
roots $\alpha_v^{(j)}(f)$ of $X^2-\lambda_v(f)X +q_v$ have modulus $|\alpha_v^{(j)}(f)| = q_v^{1/2}$. Hence
\begin{equation}\label{ramanujan}
|\alpha_v^{(j)}(f)c_v^{(k)}(\chi)q_v^{-1/2}| \leq 1,
\end{equation}
so that (say, for $\mathrm{Re}(s) > 0$)
\begin{align*}
\frac{L^{\prime}}{L}(f \times \chi, s) & = -\log(q)\sum_{n=1}^{\infty}(q^{-(s+\frac{1}{2})})^{n} \\
& \quad - 2\log(q_{P_0})\sum_{n=1}^{\infty}(\chi(w_{P_0})q_{P_0}^{-2(s+\frac{1}{2})})^{n} \\
& \quad - \sum_{v \nmid P_0 \infty}\log(q_v)\sum_{1 \leq j, k \leq 2} \sum_{n=1}^{\infty}(\alpha_v^{(j)}(f)c_v^{(k)}(\chi)q_v^{-(s+\frac{1}{2})})^{n}.
\end{align*}
Moreover, since $q_v=q^{\deg(v)}$ where $\deg(v)$ is the degree of the residue field $\mathbb{F}_v$ of $k_v$,
a calculation yields
\begin{align*}
\frac{L^{\prime}}{L}(f \times \chi, s) & = \sum_{n=1}^{\infty}\frac{a(n)}{q^{ns}} + \sum_{n=1}^{\infty}\frac{b(n)}{q^{ns}} + \sum_{n=1}^{\infty}\frac{c(n)}{q^{ns}},
\end{align*}
where $a(n):=-\log(q)q^{-n/2}$,
\[
b(n):=
\begin{cases}
-2\deg(P_0)\log(q)(\chi(w_{P_0})q_{P_0}^{-1})^{n/2\deg(P_0)},  & \textrm{if $2\deg(P_0) \mid n$}\\
0, &  \textrm{if $2\deg(P_0) \nmid n$}
\end{cases}
\]
and
\[
c(n):=-\log(q)\sum_{d\mid n}d\sum_{\substack{v \nmid P_0 \infty\\ \deg(v)=d}}\sum_{1 \leq j, k \leq 2}{(\alpha_v^{(j)}(f)c_v^{(k)}(\chi)q_v^{-\frac{1}{2}})}^{n/d}.
\]

Since $c_{f,K}(n)=a(n)+b(n)+c(n)$, then using \eqref{ramanujan} we estimate
\begin{align*}
|c_{f,K}(n)| &\leq |a(n)| + |b(n)| + |c(n)| \\
& \leq \log(q) + 2\log(q)\deg(P_0) + 4\log(q)\sum_{d\mid n} d \cdot \#\{v \nmid P_0 \infty: \deg(v)=d\}\\
& \leq  \log(q) + 2\log(q)\deg(P_0) + 4\log(q) \sum_{d\mid n} d q^{d}\\
& \leq \log(q)(1 + 2\deg(P_0) + 4\sigma(n)q^{n}),
\end{align*}
where
\[
\sigma(n):=\sum_{d\mid n}d.
\]
We have
\[
\sigma(n) = O(\log(\log(n))n).
\]
Therefore, using that $s_0 > 1/2$, we apply these bounds in \eqref{centralbound} to get
\begin{align*}
\log|L(f \times \chi, 1/2)| & \leq \frac{3m}{2h} +  h^{-1} \log(q)^{-2}\sum_{n=1}^{h}\frac{c_{f,K}(n)\log(q^{h-n})}{nq^{ns_0}}\\
& \leq  \frac{3m}{2h} + O(\log(\log(h))q^{h/2}).
\end{align*}
Finally, since $\log_q(m/2) \leq h \leq  \log_q(m/2) + 1$ (so that $q^{h/2} \leq q^{1/2}(m/2)^{1/2}$), we get
\[
|L(f \times \chi, 1/2)|  \leq \exp\left( \frac{3m}{2\log_q(m/2)} + O(\log(\log(\log_q(m/2) + 1))q^{1/2}m^{1/2}) \right).
\]
This completes the proof.
\end{proof}

\section{The function field analog of Gross's formula}

In this section we study an analog of Gross's formula \cite{G} over rational function fields
due to Papikian \cite{P2}, Wei and Yu \cite{WY}. In particular, we give an alternative expression for this formula which will be useful for
our calculations.

First, we recall some facts from the introduction. The Shimura curve $X_{P_0}$ is the disjoint union of $n$ genus zero curves $X_i$ defined over $k$.
Hence if $\mathrm{Pic}(X_{P_0})$ denotes the Picard group of $X_{P_0}$ and if $e_i$ denotes the class of degree 1 in
$\mathrm{Pic}(X_{P_0})$ corresponding to the component $X_i$, we have
\[
\mathrm{Pic}(X_{P_0})=\mathbb Z e_1 \oplus \cdots \oplus \mathbb Z e_{n}.
\]
Since a Gross point $x \in \mathrm{Gr}_{D,{P_0}}$ lies on a component $X_i$, it
determines a class $e_{x}$ in $\mathrm{Pic}(X_{P_0})$ which for notational convenience we continue to denote by $x$.
We denote the action of $\mathrm{Pic}(\mathcal{O}_D)$ on $\mathrm{Gr}_{D,{P_0}}$ by $x \mapsto x^{\sigma}$ for
$\sigma \in \mathrm{Pic}(\mathcal{O}_D)$.

The Gross height pairing
\[
\langle ~ , ~ \rangle : \mathrm{Pic}(X_{P_0}) \times \mathrm{Pic}(X_{P_0}) \rightarrow \Z
\]
is defined on generators by $\langle e_i, e_j \rangle = w_i \delta_{ij}$ and extended bi-additively to $\mathrm{Pic}(X_{P_0})$, where
$w_i:=\#(R_i^{\times})/(q-1)$.

Also, recall that $M(P_0)$ (resp.~$S(P_0)$) denotes the space of automorphic forms (resp.\ cusp forms) of Drinfeld type for $\Gamma_0(P_0)$ with the
Petersson inner product $\langle f, f \rangle_{P_0}$.

Let $\mathcal{F}(P_0)$ be an orthogonal basis for $S(P_0)$ consisting
of normalized newforms. By the Jacquet-Langlands correspondence over $k$ and the multiplicity-one theorem,
for each form $f \in \mathcal{F}(P_0)$,
there is a unique one-dimensional eigenspace $\R e_f$ in $\mathrm{Pic}(X_{P_0})\otimes_{\Z} \R$ such that $\langle e_f , e_f \rangle = 1$ and
$t_m e_f = \lambda_m(f)e_f$ for each monic polynomial $m \in A$ with $(m,P_0)=1$, where $t_m$
is the Hecke correspondence and $\lambda_m(f) \in \R$ is the eigenvalue for $f$ associated to the Hecke operator $T_m$ (see e.g. \cite[\S 2.4, \S 4.4.1]{W}).

Let $e_f \in \mathrm{Pic}(X_{P_0})\otimes_{\Z} \R$ correspond to $f \in \mathcal{F}(P_0)$ as above,
and define
\[
e^{*}:=\sum_{i=1}^{n}\frac{1}{w_i}e_i.
\]
Then an orthonormal basis for $\mathrm{Pic}(X_{P_0})\otimes_{\Z} \R$ is given by
\[
\left\{\frac{e^{*}}{\sqrt{\langle e^{*}, e^{*} \rangle}}\right\} \cup \left\{e_f : ~ f \in \mathcal{F}(P_0) \right\}.
\]

Given a character $\chi$ of $\mathrm{Pic}(\mathcal{O}_D)$ and a Gross point $x \in \mathrm{Gr}_{D, P_0}$, define
\[
c_{\chi}:=\sum_{\sigma \in \mathrm{Pic}(\mathcal{O}_D)}\overline{\chi(\sigma)}x^{\sigma} \in \mathrm{Pic}(X_{P_0})\otimes_{\Z} \C.
\]
Moreover, given a form $f \in \mathcal{F}(P_0)$,
recall that $L(f \times \chi, s)$ denotes the Rankin-Selberg $L$--function associated to $f$ and $\chi$
(normalized so that the central value occurs at $s=1/2$). Then Papikian \cite{P2}, Wei and Yu \cite[Thm.~3.3]{WY} proved
the following Gross-type formula (recall that $D$ is irreducible and $\deg(D)$ is odd)
\begin{equation}\label{PWY}
L(f \times \chi, 1/2) = \frac{\langle f, f \rangle_{P_0}}{q^{\frac{\deg(D) + 1}{2}}}
\left|\langle c_{\chi}, e_f\rangle\right|^2,
\end{equation}
where $\langle f, f \rangle_{P_0}$ is the Petterson inner product.

We now give an alternative expression for \eqref{PWY} which will be useful for our calculations. Let
$M_B^{\C}(P_0)$ be the vector space of $\C$-valued functions on $\mathrm{Pic}(X_{P_0}) \otimes_{\Z} \C$, with
inner product
\[
\langle \phi, \psi \rangle : = \sum_{i=1}^{n}w_i \phi(e_i)\overline{\psi(e_i)}.
\]
Then the map which sends a generator $e_i$ to its characteristic function $\delta_{e_i}$ induces
an isomorphism
\[
\mathrm{Pic}(X_{P_0}) \otimes_{\Z} \C \cong M_B^{\C}(P_0)
\]
defined by
\[
e=\sum_{i=1}^{n}c_ie_i \longmapsto \widetilde{e}:=\sum_{i=1}^{n}c_i\delta_{e_i}.
\]
Moreover, this map is an isometry of inner-product spaces, i.e.,
$\langle \widetilde{e}, \widetilde{e}^{\prime} \rangle = \langle e, e^{\prime} \rangle$
for any $e, e^{\prime} \in \mathrm{Pic}(X_{P_0}) \otimes_{\Z} \C$.

Let $\widetilde{f}=\widetilde{e_f}$ denote the image of $e_f$ under this isomorphism. Then an orthonormal basis for $M_B^{\R}(P_0)$ is given by
\[
\Big\{\frac{\widetilde{e}^{*}}{\sqrt{\langle \widetilde{e}^{*} , \widetilde{e}^{*} \rangle}}\Big\} \cup \{\widetilde{f}:~f \in \mathcal{F}(P_0)\}.
\]
We can now express \eqref{PWY} as
\begin{equation}\label{gf}
L(f \times \chi, 1/2)=\frac{\langle f, f \rangle_{P_0}}{q^{\frac{\deg(D) + 1}{2}}}\,
\Biggl| \sum_{\sigma \in \mathrm{Pic}(\mathcal{O}_D)}\overline{\chi(\sigma)}w_{\sigma}\widetilde{f}(x^{\sigma})\Biggr|^2,
\end{equation}
where we write $w_{\sigma}$ for $w_i$ if $x^{\sigma}$ lies in the class $e_i$.

\section{Proof of Theorem~\ref{Grossed}}

In this section we prove Theorem \ref{Grossed}.  We begin by showing that
\begin{equation}
\label{eq:Nformula}
\frac{N_{G_D, P_0, e_i}}{|G_D|} = \mu_{P_0}(e_i) + \frac{1}{w_i} \frac{1}{h(D)}\sum_{\substack{ \chi \in \mathrm{Pic}(\mathcal{O}_D)\\ \chi_{|_G}=1}}
\sum_{f \in \mathcal{F}(P_0)} \langle \widetilde{e}_i, \widetilde{f} \rangle W_{\chi, D,\widetilde{f}},
\end{equation}
where the Weyl sum is defined by
\[
W_{\chi, D, \widetilde{f}}:=\sum_{\sigma \in \mathrm{Pic}(\mathcal{O}_D)}\chi(\sigma)w_{\sigma}\widetilde{f}(x^{\sigma}).
\]

We have
\begin{equation*}
 w_i N_{G_D, P_0, e_i} = w_i \# \{ \sigma \in G_D : x_{0,D}^{\sigma} = e_i \} = \sum_{\sigma \in G_D} w_\sigma   \widetilde{e}_i(x^{\sigma}).
\end{equation*}
By decomposing the function $\widetilde{e}_i$ into a Hecke basis in $M_B^{\mathbb{R}}(P_0)$, we get
\[
\widetilde{e}_i(z)=
\frac{\langle \widetilde{e}_i, \widetilde{e}^* \rangle}{\langle \widetilde{e}^*, \widetilde{e}^* \rangle}\widetilde{e}^*(z)
+ \sum_{f \in \mathcal{F}(P_0)}\langle \widetilde{e}_i, \widetilde{f} \rangle \widetilde{f}(z).
\]
Therefore
\[
w_i N_{G_D, P_0, e_i} = \frac{\langle \widetilde{e}_i, \widetilde{e}^* \rangle}{\langle \widetilde{e}^*, \widetilde{e}^* \rangle}
\sum_{\sigma \in G_D} w_{\sigma} \widetilde{e}^*(x^{\sigma})
+ \sum_{f \in \mathcal{F}(P_0)}\langle \widetilde{e}_i, \widetilde{f} \rangle \sum_{\sigma \in G_D}w_{\sigma}\widetilde{f}(x^{\sigma}).
\]
Now, recall that the probability measure $\mu_{P_0}$ on $\mathcal{S}$ is defined by
\[
\mu_{P_0}(e_i):=\frac{w_i^{-1}}{\sum_{j=1}^{n}w_j^{-1}}
\]
where $w_i:=\#(R_i^{\times})/(q-1)$. Then, using that $\langle \widetilde{e}_i, \widetilde{e}^* \rangle = 1$ for all $i$ and
$\langle \widetilde{e}^*, \widetilde{e}^* \rangle = \sum_{j=1}^{n}w_j^{-1}$, we get
\[
\frac{\langle \widetilde{e}_i, \widetilde{e}^* \rangle}{\langle \widetilde{e}^*, \widetilde{e}^* \rangle} = w_i\mu_{P_0}(e_i).
\]
Also, we compute
\begin{equation*}
\sum_{\sigma \in G_D} w_{\sigma} \widetilde{e}^*(x^{\sigma}) =
\sum_{\sigma \in G_D} w_{\sigma} \sum_{i=1}^{n} \frac{1}{w_i} \widetilde{e}_i (x^{\sigma}) = |G_D|.
\end{equation*}
Last, by Fourier analysis we have
\[
\frac{1}{|G_D|}\sum_{\sigma \in G_D}w_{\sigma}\widetilde{f}(x^{\sigma}) =
\frac{1}{h(D)}\sum_{\substack{ \chi \in \mathrm{Pic}(\mathcal{O}_D)\\ \chi_{|_{G_D}}=1}}\sum_{\sigma \in \mathrm{Pic}(\mathcal{O}_D)}\chi(\sigma)w_{\sigma}\widetilde{f}(x^{\sigma}).
\]
Then combining these calculations yields \eqref{eq:Nformula}.

We now turn to the proof of Theorem \ref{Grossed}. By Cauchy's inequality, we have
\begin{equation}\label{Nbound}
\left|\frac{N_{G_D, P_0, e_i}}{|G_D|} - \mu_{P_0}(e_i)\right| \leq
\frac{1}{w_i} \frac{1}{h(D)}\sum_{\substack{ \chi \in \mathrm{Pic}(\mathcal{O}_D)\\ \chi_{|_{G_D}}=1}} M_1^{1/2} M_2^{1/2},
\end{equation}
where
\[
M_1:=\sum_{f \in \mathcal{F}(P_0)}|\langle \widetilde{e}_i , \widetilde{f} \rangle|^2 , \qquad M_2:=\sum_{f \in \mathcal{F}(P_0)}|W_{\chi, D,\widetilde{f}}|^2.
\]

By Bessel's inequality, we get
\[
M_1 \leq \langle \widetilde{e}_i, \widetilde{e}_i \rangle = w_i.
\]

Next, by the period formula \eqref{gf} we have
\[
|W_{\chi, D, \widetilde{f}}|^2 = \frac{q^{1/2} \lVert D \rVert^{1/2}}{\langle f, f \rangle_{P_0}} L(f \times \overline{\chi}, 1/2),
\]
which yields
\[
M_2=q^{1/2}\lVert D \rVert^{1/2} \sum_{f \in \mathcal{F}(P_0)} \frac{L(f \times \overline{\chi}, 1/2)}{\langle f, f \rangle_{P_0}}.
\]
Theorem \ref{lindelof} gives the uniform Lindel\"of bound
\[
L(f \times \overline{\chi}, 1/2) \ll_{\varepsilon} q^{m\varepsilon} \ll_{\varepsilon} \lVert P_0 \rVert^{\varepsilon} \lVert D \rVert^{\varepsilon},
\]
where we recall that $m=O(\deg(P_0) + \deg(D))$. Then applying this bound gives
\[
M_2^{1/2} \ll_{\varepsilon} C(P_0)^{1/2} q^{1/4} \lVert P_0 \rVert^{\varepsilon/2} \lVert D \rVert^{1/4 + \varepsilon/2},
\]
where
\[
C(P_0):= \sum_{f \in \mathcal{F}(P_0)} \frac{1}{\langle f, f \rangle_{P_0}}.
\]

To bound $C(P_0)$, note that by \cite[Corollary 5.6]{P1} we have $\langle f, f \rangle_{P_0} \gg_{\varepsilon} \lVert P_0 \rVert^{1-\varepsilon}$.
Moreover, since the dimension of $S(P_0)$ is equal to $h_{P_0}-1$, where
\begin{align}\label{dimensionformula}
h_{P_0}:=\frac{1}{q^2-1}\left(\lVert P_0\rVert - 1\right) + \frac{q}{2(q+1)}\left(1-(-1)^{\deg(P_0)}\right)
\end{align}
(see for example \cite[p.~740]{WY}), we get
\begin{align}\label{dimensionbound}
 \# \mathcal{F}(P_0) \ll \lVert P_0\rVert.
\end{align}
Hence
\begin{align*}
C(P_0) \ll_{\varepsilon} \lVert P_0\rVert^{\varepsilon}.
\end{align*}

Finally, by combining the preceding estimates, it follows from \eqref{Nbound} and the lower bound
\[
h(D) \gg_{\varepsilon} q^{-1}  \lVert D \rVert^{1/2-\varepsilon}
\]
that
\[
\left|\frac{N_{G_D, P_0, e_i}}{|G_D|} - \mu_{P_0}(e_i)\right| \ll_{\varepsilon} [\mathrm{Pic}(\mathcal{O}_D):G_D] q^{5/4} \lVert P_0 \rVert^{3\varepsilon/2}
 \lVert D \rVert^{-1/4 + 3\varepsilon/2}.
\]
Replacing $\varepsilon$ with $2\varepsilon/3$, we complete the proof.
\qed


\end{document}